\documentclass[12pt,a4paper]{amsart}
\usepackage{amsfonts,amssymb,amsmath,amsthm}
\usepackage{url}
\usepackage{enumerate}
\usepackage{bbm}
\usepackage{amssymb}
\usepackage{mathrsfs}
\usepackage[usenames]{xcolor}
\usepackage[left=2.5cm, top=2.5cm,bottom=2.5cm,right=2.5cm]{geometry}

\newtheorem{thrm}{Theorem}[section]
\newtheorem*{thrm*}{Theorem}
\newtheorem{thmalpha}{Theorem}

\newtheorem{lem}[thrm]{Lemma}

\theoremstyle{definition}

\newtheorem{remark}[thrm]{Remark}
\numberwithin{equation}{section}

\DeclareMathOperator{\disp}{disp}
\newcommand{\kdisp}{k\!\operatorname{-disp}}
\newcommand{\odisp}{0\!\operatorname{-disp}}

\newcommand{\N}{\ensuremath{{\mathbb N}}}

\newcommand{\R}{\ensuremath{{\mathbb R}}}

\newcommand{\Pro}{\ensuremath{{\mathbb P}}}

\newcommand{\vol}{\mathrm{vol}}

\newcommand{\eps}{\varepsilon}

\author{ Aicke Hinrichs \and Joscha Prochno \and Mario Ullrich \and Jan Vyb\'iral}
\date{\today}
\address[Aicke Hinrichs]{Institut f\"ur Analysis\\
Johannes Kepler Universit\"at Linz\\
Altenbergerstrasse 69\\
4040 Linz\\
Austria}
\email{aicke.hinrichs@jku.at}
\address[Mario Ullrich]{Institut f\"ur Analysis\\
Johannes Kepler Universit\"at Linz\\
Altenbergerstrasse 69\\
4040 Linz\\
Austria}
\email{mario.ullrich@jku.at}
\address[Joscha Prochno]{School of Mathematics \& Physical Sciences\\
University of Hull\\
Cottingham Road\\
Hull HU6 7RX\\
United Kingdom}
\email{j.prochno@hull.ac.uk}
\address[Jan Vyb\'iral]{Department of Mathematics\\
Faculty of Nuclear Sciences and Physical Engineering\\
Czech Technical University\\
Trojanova 13\\
12000 Praha \\
Czech Republic}
\email{jan.vybiral@fjfi.cvut.cz}
\thanks{We would like to express our gratitude to the Erwin Schr\"odinger International Institute for Mathematics and Physics for its hospitality during
the programme on ``Tractability of High Dimensional Problems and Discrepancy'', where some part of this research was carried out. We also gratefully
acknowledge the support of the Oberwolfach Research Institute for Mathematics, where initial discussion were held during the workshop
``Perspectives in High-Dimensional Probability and Convexity''. JP has been supported by a Visiting International Professor Fellowship from the Ruhr University Bochum. JV was supported by the grant P201/18/00580S
of the Grant Agency of the Czech Republic and by the Neuron Fund for Support of Science.}


\begin{document}

\title[Minimal $k$-Dispersion]{The minimal $k$-dispersion of point sets in high-dimensions}

\begin{abstract}
In this manuscript we introduce and study an extended version of the minimal dispersion of point sets, which has recently attracted considerable attention.
Given a set $\mathscr P_n=\{x_1,\dots,x_n\}\subset [0,1]^d$ and $k\in\{0,1,\dots,n\}$, we define the $k$-dispersion to be the volume of the largest box
amidst a point set containing at most $k$ points. The minimal $k$-dispersion is then given by the infimum over all possible point sets of cardinality $n$.
We provide both upper and lower bounds for the minimal $k$-dispersion that coincide with the known bounds for the classical minimal dispersion
for a surprisingly large range of $k$'s.
\end{abstract}
\maketitle



\section{Introduction and main results}

A classical problem in computational geometry and complexity asks for the size of the largest empty and axis-parallel box given a point configuration in the cube $[0,1]^2$. From a complexity point of view this \emph{maximum empty rectangle problem} was already studied by Naamad, Lee, and Hsu \cite{NLH1984} who provided an $\mathcal O(n^2)$-time algorithm as well as an $\mathcal O(n\log^2(n))$-expected-time algorithm (when the points are drawn independently and uniformly at random) to actually find such a rectangle. Further results in this direction have been obtained by Chazelle, Drysdale, and Lee in \cite{CDL1986}.

In the last decade the study of high-dimensional (geometric) structures has become increasingly important and it has been realized by now that the presence of high dimensions forces a certain regularity in the geometry of the space while, on the other hand, it unfolds various rather unexpected phenomena. The maximum empty rectangle problem studied in \cite{NLH1984} has its natural multivariate counterpart. Let $d,n\in\N$ and denote by $\mathscr B^d_{\text{ax}}$ the collection of axis-parallel boxes in $[0,1]^d$. Then the ($n$-th) minimal dispersion is defined to be the quantity
\begin{equation}\label{eq:min:1}
\disp^*(n,d) := \inf_{\mathscr P\subset[0,1]^d\atop{\#\mathscr P=n}}\,\sup_{\mathbb B\in\mathscr B_{\text{ax}}\atop{\mathbb B\cap \mathscr  P=\emptyset}} \vol_d(\mathbb B),
\end{equation}
where $\#$ denotes the cardinality of a set and $\vol_d(\cdot)$ the $d$-dimensional Lebesgue measure. In other words, the minimal dispersion depicts the size of the largest empty, axis-parallel box amidst \emph{any} point set of cardinality $n$ in the $d$-dimensional cube. The interest in this notion is in parts motivated by applications in approximation theory, more precisely, in problems concerning the approximation of high-dimensional rank one tensors \cite{BDDG2014,NR2016} and in Marcinkiewicz-type discretizations of the uniform norm of multivariate trigonometric polynomials \cite{T2018}. Even though the minimal dispersion is conceptually quite simple and the problem of determining (or estimating) its order attracted considerable attention in the past $3$ years (see, e.g., \cite{AHR17, DJ2013, K2018, R2018,S2018,U2018, UV2018}), its behavior, simultaneously in the number of points $n$ and in the dimension $d$, is still not completely understood. 

Let us briefly describe the current state of the art. Aistleitner, Hinrichs, and Rudolf proved in \cite[Theorem 1]{AHR17} that, for all $d,n\in\N$,
\[
\disp^*(n,d) \geq \frac{\log_2(d)}{4(n+\log_2(d))}\,.
\]
In particular, this shows that the volume of the largest empty box increases with the dimension $d$. In the same paper, the authors communicate an upper bound, which is attributed to Larcher (see \cite[Section 4]{AHR17}) and shows that
\[
\disp^*(n,d) \leq \frac{2^{7d+1}}{n}\,.
\]  
This improves upon a bound of Rote and Tichy \cite[Proposition 3.1]{RT1996} when $d\geq 54$. Under the assumption that the number of points satisfies $n> 2d$, it was recently proved by Rudolf \cite[Corollary 1]{R2018} that
\[
\disp^*(n,d) \leq  \frac{4d}{n} \log_2\Big(\frac{9n}{d}\Big)\,.
\]

In contrast to the probabilistic methods, several authors then provided explicit constructions of sets with small dispersion and small number of points. 
For example, Krieg \cite[Theorem]{K2018} shows that for $\varepsilon\in(0,1)$ and $d\ge 2$ there is a sparse grid with the number of points
bounded by $(2d)^{\log_2(1/\varepsilon)}$ and dispersion at most $\varepsilon$. 
Temlyakov proved in \cite[Theorem 2.1 and Theorem 4.1]{T2018'} that Fibonacci and 
Frolov point sets achieve the dispersion of optimal order $n^{-1}$, 
but without paying extra attention to the dependence on $d$, 
see also~\cite{U2018b}.

An essential breakthrough was achieved by Sosnovec in \cite[Theorem 2]{S2018}. He provided a randomized construction of a set with at most
$c_\varepsilon \log_2(d)$ points in $[0,1]^d$ with dispersion at most $\varepsilon\in(0,1/4]$. Here, $c_\varepsilon\in(0,\infty)$
is a quantity depending only on $\varepsilon.$ The $\varepsilon$-dependence was then refined by the third and fourth author in \cite[Theorem 1]{UV2018}. More precisely, for every $\varepsilon\in (0,1/2)$ and $d\ge 2$, they provided a randomized construction of a point set $\mathscr P$ with
$$
\#{\mathscr P}\le 2^7\frac{(1+\log_2(\varepsilon^{-1}))^2}{\varepsilon^2}\log_2 (d)
$$
and dispersion at most $\varepsilon$. This result can be also reformulated as 
\begin{equation}\label{eq:UV}
\disp^*(n,d)\le c\log_2(n)\sqrt{\frac{\log_2(d)}{n}}
\end{equation}
for $n,d\ge 2$ and some absolute constant $c\in(0,\infty)$.

In this manuscript, we generalize the notion of dispersion by introducing a quantity, which  measures the size of the largest box amidst a point 
set containing \emph{at most} $k$ points. The minimal dispersion \eqref{eq:min:1} then corresponds to the case $k=0$.

For $d,n\in\N$ and $k\in\N\cup \{0\}$, we define the $k$-dispersion of a point set 
$\mathscr P_n=\{x_1,\dots,x_n\}\subset [0,1]^d$ to be the quantity
\[
\kdisp(\mathscr P_n,d) := \sup\Big\{\vol_d(\mathbb B)\,\colon\, \mathbb B\in \mathscr B_{\text ax}^d \text{ with }\# (\mathscr P_n\cap \mathbb B)\leq k\Big\},
\]
where $\mathscr B_{\text ax}^d$ is the collection of all axis-parallel boxes inside the $d$-dimensional cube $[0,1]^d$. 
The minimal $k$-dispersion 
is defined to be the infimum 
over all possible point sets of cardinality $n$, i.e.,
\[
\kdisp^*(n,d) := \inf_{\substack{\mathscr P_n\subset [0,1]^d\\ \#\mathscr P_n=n}} \kdisp(\mathscr P_n,d).
\]

Our first main result provides an upper bound on the minimal $k$-dispersion.

\begin{thmalpha}\label{thm:main 1}
There exists a constant $C\in(0,\infty)$ such that for any $d\geq 2$ and all $k,n\in\N$ with $k<n/2$, we have  
\begin{equation}\label{eq:dispk}
\kdisp^*(n,d) \,\le\, C\,\max\left\{\log_2 (n)\sqrt{\frac{\log_2 (d)}{n}},\,
	k\,\frac{\log_2(n/k)}{n}\right\}.
\end{equation}
\end{thmalpha}

\begin{remark} The minimal $k$-dispersion is easily seen to be non-decreasing in $k$. On the other hand,
a comparison of \eqref{eq:UV} and \eqref{eq:dispk} reveals that the upper bound of minimal dispersion and minimal $k$-dispersion
are of the same order for a large range of $k$'s. Indeed, if $k=k(n,d)$ increases with the number of points $n$ and the dimension $d$ while satisfying  
\[
k(n,d) \leq c \sqrt{n\cdot\log_2(d)}\,,
\]
for some absolute constant $c\in(0,\infty)$, then \eqref{eq:UV} and \eqref{eq:dispk} provide the same order in $n$ and $d$. Motivated by this result, we conjecture that \eqref{eq:UV} actually offers a lot of space for improvement.
\end{remark}

The second result establishes the following lower bound on the minimal $k$-dispersion.

\begin{thmalpha}\label{thm:main 2}
Let $k,n,d\in\N$. Then
\[
\kdisp^*(n,d) \,\ge\, \frac18\,\min\left\{1, \frac{k+\log_2(d)}{n}\right\}.
\]
\end{thmalpha} 



\section{Proof of Theorem \ref{thm:main 1} -- the upper bound}\label{sec:thm1-multi}

We shall present here the proof for the upper bound of the minimal $k$-dispersion.
In the proof, we modify the ideas developed in \cite{S2018} and \cite{UV2018} to our setting. 

\subsection{The idea of proof}

Before we start let us briefly discuss the strategy of the proof. For every $\varepsilon\in(0,1/4)$,
we construct a set ${\mathbb X}=\{x^1,\dots,x^n\}\subset[0,1]^d$ with a small number of elements and small $k$-dispersion.
This set is constructed (similarly to \cite{S2018} and \cite{UV2018}) by random sampling from a discrete mash in $[0,1]^d$.
To allow for independence across the steps of this sampling, it might happen that some points of the mash are actually
sampled more than once. Such points were naturally discarded in \cite{UV2018}, because they could not influence if a box ${\mathbb B}\subset[0,1]^d$
intersects ${\mathbb X}$, or not.

Here, we need to keep a more detailed track of the intersection of boxes with ${\mathbb X}$. Therefore, we allow for repeated sampling,
and ${\mathbb X}$ will actually be a multiset. In Section \ref{subsec:multiset}, we argue in detail that this modification does not alter
the minimal $k$-dispersion. It is then not difficult to see that each box of large enough volume contains indeed at least $k$ points
from the multiset ${\mathbb X}$ with high probability. Unfortunately, this is still not enough to apply the union bound, as there are infinitely many
boxes with large volume. We will therefore divide them into finitely many groups (called $\Omega_{m}(s,p)$ later on), and show that even the intersection
of all boxes included in $\Omega_m(s,p)$ still contains at least $k$ points from ${\mathbb X}$ with high probability. At long last, we can apply the union
bound over all admissible parameter pairs $(s,p)$.

\subsection{A random multiset}\label{subsec:multiset}

Let $\varepsilon\in (0,1/4)$ be fixed and put $m=\lceil\log_2(1/\eps)\rceil$. 
We define a one-dimensional set via
\[
\mathbb M_m:=\left\{\frac{1}{2^{m}},\dots,\frac{2^{m}-1}{2^{m}}\right\}\subset [0,1].
\]
For $n\in\N$, we construct a random multiset 
$\mathbb X=\{x^1,\dots,x^n\}$ by sampling independently and uniformly at random from 
the points in $\mathbb M_m^d$, where $n$ will later be the number of points in Theorem \ref{thm:main 1}. 
Note that a multiset $\mathbb X$ in the cube $[0,1]^d$ is naturally identified 
with the \emph{multiplicity function} $\mathbb X\colon [0,1]^d\to \N_0$, 
where $\mathbb X(z)$ gives the multiplicity of $z\in [0,1]^d$ in $\mathbb X$.
For ${\mathbb B}\in {\mathscr B}_{ax}$, we define
\[
\#\bigr(\mathbb X \cap {\mathbb B}\bigl) \,:=\, \sum_{z\in {\mathbb B}} \mathbb X(z)\quad\text{and}\quad
\#{\mathbb X} \,:=\, \sum_{z\in[0,1]^d} \mathbb X(z)\,,
\]
and the $k$-dispersion of the multiset ${\mathbb X}$ as
$$
\kdisp_m({\mathbb X},d):= \sup\Big\{\vol_d(\mathbb B)\,\colon\, \mathbb B\in \mathscr B_{\text ax}^d \text{ with }\# ({\mathbb X}\cap \mathbb B)\leq k\Big\}.
$$
Finally, we take the infimum over all possible multisets of cardinality $n$ and obtain
\[
\kdisp_m^*(n,d) := \inf_{\substack{{\mathbb X}\subset [0,1]^d\\ \#{\mathbb X}=n}} \kdisp_m({\mathbb X},d).
\]
As each classical set ${\mathscr P}\in[0,1]^d$ is also a multiset (with the multiplicity function bounded by one), we immediately obtain that
\[
\kdisp_m^*(n,d)\le \kdisp^*(n,d).
\]
On the other hand, if ${\mathbb X}=\{x^1,\dots,x^n\}\subset [0,1]^d$ is a multiset,
then we consider the sets $\{x^1+\xi^1,\dots,x^n+\xi^n\}$ with $\|\xi^j\|_\infty\le \delta$ (for $\delta\in(0,\infty)$), where $\xi^1,\dots,\xi^n$ are independent random vectors which are uniformly distributed over $[-\delta,\delta]^d$. If we then let $\delta\to 0$, it follows that $\kdisp^*(n,d)\le \kdisp_m^*(n,d)$.

\subsection{The partitioning scheme}\label{subsec:partition}

We now introduce a set $\Omega_m$ containing all those boxes $\mathbb B$ with
`large' volume. For $m\in\N$, we define  
\[
\Omega_m:=\Big\{\mathbb B=I_1\times\dots\times I_d\subset[0,1]^d\,\colon\,\vol_d(\mathbb B)>\frac{1}{2^m}\Big\}\,.
\]
As already described before, our approach will later be based on a union bound over all the boxes $\mathbb B\in\Omega_m$. As there are infinitely many of those boxes, we first divide $\Omega_m$ into finitely many `suitable' subsets. 
This is done as follows: for $s=(s_1,\dots,s_d)\in\{0,1,\dots,2^{m}-1\}^d$ and 
$p=(p_1,\dots,p_d)\in\{1/2^{m},\dots,1-1/2^{m}\}^d$, we define the collection 
$\Omega_m(s,p)$ of subsets of $\Omega_m$ to be 
\begin{align*}
\Omega_m(s,p)&:=\bigg\{\mathbb B=I_1\times\dots\times I_d\in\Omega_m \,\colon\, 
\forall \ell\in\{1,\dots,d\}: \frac{s_\ell}{2^{m}}<\vol(I_\ell) \le \frac{s_\ell+1}{2^{m}}\\
&\qquad\qquad \text{and}\quad\inf I_\ell\in \Big[p_\ell-\frac{1}{2^{m}},p_\ell\Big)\bigg\}.
\end{align*}
We first observe that $\Omega_m(s,p)=\emptyset$ if the choice of $s$ does not 
allow $\Omega_m(s,p)$ to contain any box $\mathbb{B}$ with $\vol_d(\mathbb B)>2^{-m}$. 
This holds, e.g., if $s_{\ell_0}=0$ for some $\ell_0\in\{1,\dots,d\}$.
We define the index set
\[
\mathbb I_m \,:=\, \Bigl\{(s,p)\,\colon\, \Omega_m(s,p)\neq\emptyset\Bigr\},
\]
which contains those indices $(s,p)$ that are needed for the following considerations, 
and we bound its cardinality. Let
\[
A_m(s) := \# \Big\{\ell\in\{1,\dots,d\}\,\colon\, s_\ell <  2^m-1 \Big\}
\]
and observe that, by definition, any $\mathbb B\in\Omega_m(s,p)$ must satisfy
$$
\frac{1}{2^m} \,<\, \vol_d(\mathbb B) 
\,\le\, \prod_{\ell=1}^d\frac{s_\ell+1}{2^{m}} 
\,\le\, \Bigl(1-\frac{1}{2^{m}}\Bigr)^{A_m(s)}.
$$
This is a contradiction if $A_m(s)>\log(2)\, m 2^m$. 
Therefore, $\Omega_m(s,p)\neq\emptyset$ implies that 
\begin{equation*}
A_m(s) \,\le\, \min\Bigl\{\lfloor\log(2)\, m 2^m\rfloor,\, d\Bigr\} \,=:\, A_m,
\end{equation*}
i.e., there are at most $A_m$ choices of $\ell$ with $s_\ell<2^m-1$. 
Clearly, there are at most $\binom{d}{A_m} 2^{m A_m}$ 
choices for $s\in\{0,1,\dots,2^{m}-1\}^d$ with $A_m(s)\le A_m$. 
Moreover, for given $s$, 
there are at most $2^{m A_m(s)}$ choices for $p$ with 
$\Omega_m(s,p)\neq\emptyset$. 
This follows from the fact that for each $\ell\in\{1,\dots,d\}$, 
we have at most $2^m-1$ choices for $p_\ell$ (by definition) and, 
if $s_{\ell_0}=2^m-1$ for some $\ell_0\in\{1,\dots,d\}$, 
then we have $\Omega_m(s,p)=\emptyset$ unless $p_{\ell_0}=2^{-m}$. 
For other $p_{\ell_0}$ the boxes cannot be contained in the unit cube.

For $m$ such that $A_m<d$, we obtain
\begin{equation*}
\begin{split}
\#\mathbb I_m \,&<\, \binom{d}{A_m} \, 2^{2m A_m}
\,<\, \biggl(\frac{ed}{A_m}\cdot 2^{2m}\biggr)^{A_m} \\
&<\, \biggl(\frac{4d2^{m+1}}{m}\biggr)^{A_m} 
\,\le\, \exp\Bigl(m2^{m}\log(2^{m+3}d)\Bigr),
\end{split}
\end{equation*}
where we used that $\log(2)m 2^{m-1}< A_m\le \log(2) m 2^{m}$ for $m\in\N$ and 
$e/\log(2)< 4$.

On the other hand, if $A_m=d$, i.e., if $d<\log(2)m2^m$, then we obtain
$$
\#\mathbb I_m\le 2^{md}\cdot 2^{md}\le \exp\Big(\log^2(2)\cdot2m2^m\Big).
$$

Therefore, for arbitrary $m\in\N$, we obtain
\begin{equation}\label{eq:card}
\#\mathbb I_m \,\le\, \exp\Bigl(m2^{m}\log(2^{m+3}d)\Bigr).
\end{equation}
\medskip

\subsection{The proof}

We shall now present the proof of the upper bound on the minimal 
$k$-dispersion. We do this by proving that our random multiset has small 
$k$-dispersion with positive probability, which proves the existence of a 
`good' multiset.

The following result is from \cite[Lemma 3]{UV2018}. Note that 
it is stated there in a different way, but (the end of) its proof clearly shows 
this variant. 
For $(s,p)\in\mathbb I_m$, let
\[
\mathbb B_m(s,p)  \,:=\, \bigcap_{\mathbb B\in \Omega_m(s,p)} \mathbb B
\,=\, \prod_{\ell=1}^d \Bigl[p_\ell,p_\ell+\frac{s_\ell-1}{2^{m}}\Bigr].
\]

\begin{lem}\label{lem:discrete} 
Let $m\in\N$, $(s,p)\in\mathbb I_m$ and 
$z$ be uniformly distributed in $\mathbb M_m^d$. Then
\begin{equation*}
\Pro\big(z\in \mathbb B_m(s,p)\big)\ge \frac{1}{2 ^{m+4}}\,.
\end{equation*}
\end{lem}

\medskip

For the random multiset as constructed in Section~\ref{subsec:multiset}, we now estimate the probability that the number of points in 
$\mathbb X\cap \mathbb B_m(s,p)$ does not exceed $k\in\N$, where $k<\frac{n}{2}$ (for the case $k=0$ see~\cite{UV2018}).
Let us consider two cases.
\vskip 1mm
\emph{Case 1:} Assume that $\Pro\big(x^1\in \mathbb B_m(s,p)\big)\le 1/2$. 
Then we use Lemma~\ref{lem:discrete} and obtain the estimate
\[\begin{split}
\Pro\big(\#(\mathbb X & \cap \mathbb B_m(s,p))\le k\big) 
\,=\, \sum_{\ell=0}^k {n\choose \ell} \Pro\big(x^1\in \mathbb B_m(s,p)\big)^{\ell}\ \Pro\big(x^1\not\in \mathbb B_m(s,p)\big)^{n-\ell}\\
&\le\, (k+1){n\choose k}\Pro\big(x^1\not\in\mathbb  B_m(s,p)\big)^{n}
\,\le\, 2k \,\frac{n^k}{k!}\,\Bigl(1-\Pro\big(x^1\in\mathbb B_m(s,p)\big)\Bigr)^n\\
&\le\, \frac{2n^k}{(k-1)!}\,\Bigl(1-\frac{1}{2^{m+4}}\Bigr)^n
\,\le\, \frac{2n^k}{(k-1)!}\, \exp\Big(-\frac{n}{2^{m+4}}\Big).
\end{split}\]
\vskip 1mm
\emph{Case 2:} Assume $\Pro(x^1\in\mathbb  B_m(s,p))>1/2$. Then
\begin{align*}
\Pro(\#(\mathbb X & \cap \mathbb B_m(s,p))\le k) 
\,=\, \sum_{\ell=0}^k {n\choose \ell} \Pro\big(x^1\in \mathbb B_m(s,p)\big)^{\ell}\ \Pro\big(x^1\not\in\mathbb  B_m(s,p)\big)^{n-\ell}\\
&\le\, (k+1){n\choose k} \frac{1}{2^{n-k}} 
\,\le\, \frac{2n^k}{(k-1)!\, 2^{n-k}} 
\,\le\, \frac{2n^k}{(k-1)!}\exp\Big(-\frac{n}{2^{m+4}}\Big),
\end{align*}
where the last inequality follows from the fact that
\[
(n-k)\log(2) > \frac{n}{2}\log(2)\ge \frac{n}{2^{4}}\ge \frac{n}{2^{m+4}}.
\]
\vskip 2mm
Putting both cases together, we see that
\begin{align}\label{ineq:prob}
\Pro\big(\#(\mathbb X\cap \mathbb B_m(s,p))\le k\big)
\,\le\, \frac{2n^k}{(k-1)!}\exp\Big(-\frac{n}{2^{m+4}}\Big).
\end{align}
\smallskip

Recall from Section~\ref{subsec:partition} that 
\[
\Omega_m \,=\, \bigcup_{(s,p)\in\mathbb I_m} \Omega_m(s,p).
\]
Combining the upper bound \eqref{eq:card} on the cardinality of $\mathbb I_m$ 
with the estimate in Lemma~\ref{lem:discrete}, we obtain by a union bound that 
\begin{eqnarray*}
\Pro\big(\exists \mathbb B\in\Omega_m\,:\,\#(\mathbb X\cap \mathbb B)\le k\big) 
& \le &\sum_{(s,p)\in\mathbb I_m} \Pro\big(\exists \mathbb B\in\Omega_m(s,p)\,:\,\#(\mathbb X\cap \mathbb B)\le k\big)\\
&\le&\sum_{(s,p)\in\mathbb I_m} \Pro\big(\#(\mathbb X\cap \mathbb B_m(s,p))\le k\big)\\
&<& \frac{2n^k}{(k-1)!}\,\exp\Bigl(m2^{m}\log(2^{m+3}d)-n2^{-m-4}\Bigr)\,.
\end{eqnarray*}
The last expression will be smaller than or equal to $1$ if and only if 
\begin{equation}\label{eq:condition}
\begin{split}
n2^{-m-4} \,&\geq m2^{m}\log(2^{m+3}d) +\log\Big(\frac{2n^k}{(k-1)!}\Big) \\
\,&= \, m2^{m}\log(2^{m+3}d) + k\, \log\Big(c_k \frac{n}{k}\Big),
\end{split}
\end{equation}
with $c_k:=k\bigl(\frac{2}{(k-1)!}\bigr)^{1/k}$. Note that by Stirling's formula,  
$c_k\uparrow e$ as $k\to\infty$.

To guarantee \eqref{eq:condition}, it is enough to assume that
\[
n\geq m 2^{2m+5}\log(2^{m+3}d) \qquad\text{and}\qquad
n\geq 2^{m+5}\,k\, \log\Big(e\frac{n}{k}\Big). 
\]
It is easy to prove that the second inequality is implied by 
$n\ge k m 2^{m+9}> e k (m+5) 2^{m+5}$.
Hence, we find an $n\in\N$ with \eqref{eq:condition} such that
\[
n \,\le\, C\,m\, 2^m\, \max\Bigl\{2^{m}\,\log(2^m d),\, k  \Bigr\}
\]
for some constant $C\le2^9$.
This ensures that there exists a realization of the multiset $\mathbb X$ 
with cardinality $n$ 
such that, for all boxes $\mathbb B$ with $\vol_d(\mathbb B)>2^{-m}$, we have 
$\#(\mathbb B\cap \mathbb X)>k$. 

Using the argument of Section \ref{subsec:multiset}, we obtain
\begin{align*}
N(2^{-m},d) & := \min \Big\{N\in\N\,:\, \kdisp(N,d)\leq 2^{-m} \Big\} \\
& \leq C\,m\, 2^m\, \max\Bigl\{2^{m}\,\log(2^m d),\, k  \Bigr\}.
\end{align*} 

To finish the proof, let $\varepsilon\in(0,\frac{1}{4})$ and denote by 
$m:=m_\varepsilon\in\N$ the unique integer satisfying 
\[
\frac{1}{2^m}\le\varepsilon< \frac{1}{2^{m-1}}\,,
\] 
i.e., $m=\lceil\log_2(1/\varepsilon)\rceil$. 
By this choice of $m$,
\begin{align*}
 m 2^{2m}\log(2^{m}d) \,<\, c_1\,\log_2(d)\bigg(\frac{\log_2(1/\varepsilon)}{\varepsilon}\bigg)^2,
\end{align*}
and
\[
m 2^{m}\,k \,<\, c_2\, \frac{k\cdot\log_2(1/\varepsilon)}{\varepsilon}
\]
for some constants $c_1,c_2\in(0,\infty)$.
This means that, since $N(\cdot,d)$ is decreasing in the first argument,
\[
N(\varepsilon,d) 
\,\le\, C\cdot \max\left\{\log_2(d)\bigg(\frac{\log_2(1/\varepsilon)}{\varepsilon}\bigg)^2,\, 
\frac{k\cdot\log_2(1/\varepsilon)}{\varepsilon} \right\}
\]
for some constant $C\in(0,\infty)$.
We therefore conclude that 
\[
\kdisp(n,d) \,\le\, C'\,\max\left\{\log_2(n)\sqrt{\frac{\log_2(d)}{n}},\,
	\frac{k\cdot\log_2(n/k)}{n}\right\}
\]
with an absolute constant $C'\in(0,\infty)$.



\section{Proof of Theorem \ref{thm:main 2} -- the lower bound}\label{sec:thm2}

The proof is very much inspired by the proof of the lower bound on the dispersion 
of Aistleitner, Hinrichs and Rudolf \cite{AHR17}. 
We recall their argument in a slightly modified form.

Given a point set $\mathscr P_n=\{x_1,\dots,x_n\}$ with 
$x_i=(x_{i,1},\dots,x_{i,d})\in[0,1]^d$, 
we define the matrix $A=A(\mathscr P_n)\in\R^{n\times d}$ by
\[
A_{i,j} \,=\, \begin{cases}
1 & :\, x_{i,j} \ge 1/2\,; \\
0 & \,\text{ otherwise},
\end{cases}
\]
with $i=1,\dots,n$ and $j=1,\dots,d$. 
Note that 
if $A$ contains two equal columns, then the projection of the point set on 
the two coordinates corresponding to these columns is contained in the 
union of the lower-left and the upper-right quarter of the unit square. 
Therefore, the dispersion is at least $1/4$.
Likewise, if two columns $c_1,c_2\in\{0,1\}^{n}$ of $A$ satisfy $c_1=1-c_2$, 
then the projection is contained in the upper-left and the lower-right quarter 
and the dispersion is at least $1/4$.
Recall that $A$ has $d$ columns.
It is clear from the pigeon hole principle that there must be two columns 
that satisfy one of the above conditions whenever $d>2^{n-1}$.
This implies
\[
\odisp^*\!\big(\lceil\log_2 d\rceil, d\big) \,\ge\, 1/4.
\]

We now consider the $k$-dispersion for $k\ge1$. 
Following the above arguments, if there are two columns $c_1,c_2$ of $A$ that 
agree (or disagree) in all but $k$ entries, then there exists a box of 
volume $1/4$ that contains at most $k$ points. 
Again, from the pigeon hole principle
(and just ignoring $k$ rows), we obtain that such columns exist 
whenever $d>2^{n-k-1}$. This implies
\begin{equation}\label{eq:init}
\kdisp^*\!\big(k+\lceil\log_2 d\rceil, d\big) \,\ge\, 1/4.
\end{equation}

Finally, note that Lemma~1 from \cite{AHR17} holds also for the $k$-dispersion, 
i.e., for all $n,k,d,\ell\in\N$ we have
\begin{equation}\label{eq:rec}
\kdisp^*(n,d) \,\ge\, \frac{(\ell+1)\, \kdisp^*(\ell,d)}{n+\ell+1}.
\end{equation}

From this we conclude Theorem \ref{thm:main 2}.

\begin{proof}[Proof of Theorem \ref{thm:main 2}]
For $n\le k+\log_2(d)$, we use \eqref{eq:init} and obtain 
$\kdisp^*(n,d)\ge1/4\ge1/8$. For $n> k+\log_2(d)$, we use \eqref{eq:rec} 
with $\ell=k+\lceil\log_2(d)\rceil\le n$ and obtain
\[
\kdisp^*(n,d) \,\ge\, \frac{(\ell+1)\, \kdisp^*(\ell,d)}{n+\ell+1}
\,\ge\, \frac14\,\frac{(\ell+1)}{n+\ell+1}
\,\ge\, \frac{k+\log_2(d)}{8n}.
\]
\end{proof}

\begin{remark}
As the method to obtain \eqref{eq:init} is clearly related to packing numbers 
on the discrete cube $\{0,1\}^n$ with respect to the Hamming metric, 
one could try to apply more involved methods to obtain better bounds.
For example, the well-known \emph{sphere-packing bound} 
(also known as Hamming bound), see \cite[Theorem 5.2.7]{Lint1992}, states that the maximal size of a $k$-packing 
of the cube, say $M(n,k)$, satisfies
\[
M(n,k) \,\le\, \frac{2^n}{\sum_{t=0}^{\lfloor(k-1)/2\rfloor}\binom{n}{t}}.
\]
However, using this bound does not lead to any significant improvement.
\end{remark}

\bibliographystyle{abbrv}
\bibliography{k-dispersion}

\begin{thebibliography}{10}

\bibitem{AHR17}
C.~Aistleitner, A.~Hinrichs, and D.~Rudolf.
\newblock On the size of the largest empty box amidst a point set.
\newblock {\em Discrete Appl. Math.}, 230:146--150, 2017.

\bibitem{BDDG2014}
M.~Bachmayr, W.~Dahmen, R.~DeVore, and L.~Grasedyck.
\newblock Approximation of high-dimensional rank one tensors.
\newblock {\em Constr. Approx.}, 39(2):385--395, Apr 2014.

\bibitem{CDL1986}
B.~Chazelle, R.~L. Drysdale, and D.~T. Lee.
\newblock Computing the largest empty rectangle.
\newblock {\em SIAM J. Comput.}, 15(1):300--315, 1986.

\bibitem{DJ2013}
A.~Dumitrescu and M.~Jiang.
\newblock On the largest empty axis-parallel box amidst $n$ points.
\newblock {\em Algorithmica}, 66(2):225--248, Jun 2013.

\bibitem{K2018}
D.~Krieg.
\newblock On the dispersion of sparse grids.
\newblock {\em J. Complexity}, 45:115 -- 119, 2018.

\bibitem{NLH1984}
A.~Naamad, D.~Lee, and W.-L. Hsu.
\newblock On the maximum empty rectangle problem.
\newblock {\em Discrete Appl. Math.}, 8(3):267 -- 277, 1984.

\bibitem{NR2016}
E.~Novak and D.~Rudolf.
\newblock Tractability of the approximation of high-dimensional rank one
  tensors.
\newblock {\em Constr. Approx.}, 43(1):1--13, Feb 2016.

\bibitem{RT1996}
G.~Rote and R.~Tichy.
\newblock Quasi-monte-carlo methods and the dispersion of point sequences.
\newblock {\em Mathematical and Computer Modelling}, 23(8):9 -- 23, 1996.

\bibitem{R2018}
D.~Rudolf.
\newblock {\em An upper bound of the minimal dispersion via delta covers}.
\newblock Contemporary Computational Mathematics - a Celebration of the 80th
  Birthday of Ian Sloan. Springer-Verlag, 2018.

\bibitem{S2018}
J.~Sosnovec.
\newblock A note on minimal dispersion of point sets in the unit cube.
\newblock {\em European J. Combin.}, 69:255 -- 259, 2018.

\bibitem{T2018'}
V.~Temlyakov.
\newblock Dispersion of the {F}ibonacci and the {F}rolov point sets.
\newblock {\em preprint}, 2017.

\bibitem{T2018}
V.~Temlyakov.
\newblock Universal discretization.
\newblock {\em J. Complexity}, 47:97--109, Aug 2018.

\bibitem{U2018b}
M.~Ullrich.
\newblock A note on the dispersion of admissible lattices.
\newblock {\em arXiv:1710.08694}.

\bibitem{U2018}
M.~Ullrich.
\newblock A lower bound for the dispersion on the torus.
\newblock {\em Math. Comput. Simulation}, 143:186--190, 2018.

\bibitem{UV2018}
M.~Ullrich and J.~Vyb\'iral.
\newblock An upper bound on the minimal dispersion.
\newblock {\em J. Complexity}, 45:120 -- 126, 2018.

\bibitem{Lint1992}
J.~van Lint.
\newblock {\em Introduction to coding theory}, volume~86 of {\em Graduate Texts
  in Mathematics}.
\newblock Springer, 1992.

\end{thebibliography}

\end{document}